\documentclass[11pt]{amsart}

\theoremstyle{plain}
\newtheorem{theorem}                {Theorem}      [section]
\newtheorem*{theorem*}                {Theorem \ref{thm:appl}}
\newtheorem{proposition}  [theorem]  {Proposition}
\newtheorem{corollary}    [theorem]  {Corollary}
\newtheorem{lemma}        [theorem]  {Lemma}

\theoremstyle{definition}

\newtheorem{remark}       [theorem]  {Remark}

\numberwithin{equation}{section}

\DeclareMathOperator{\trace}{trace}
\DeclareMathOperator{\grad}{grad}
\DeclareMathOperator{\Span}{span}
\DeclareMathOperator{\Imag}{Im}

\begin{document}

\title[Biconservative surfaces in the $4$-dimensional hyperbolic space]
{Intrinsic characterizations of biconservative surfaces in the $4$-dimensional hyperbolic space}

\author{Simona Nistor, Mihaela Rusu}

\address{Faculty of Mathematics, Al. I. Cuza University of Iasi,
Blvd. Carol I, 11 \\ 700506 Iasi, Romania} \email{nistor.simona@ymail.com}

\address{Faculty of Mathematics, Al. I. Cuza University of Iasi,
Blvd. Carol I, 11 \\ 700506 Iasi, Romania} \email{mihaelarussu10@yahoo.com}

\begin{abstract}
In this paper, we extend the investigation of biconservative surfaces with parallel normalized mean curvature vector fields ($PNMC$) in the $4$-dimensional space forms, focusing on the hyperbolic space $\mathbb{H}^4$, the last remaining case to explore. We establish that an abstract surface admits a $PNMC$ biconservative immersion in $\mathbb{H}^4$ if and only if it satisfies a certain intrinsic condition; if such an immersion exists, it is unique. We further analyze these abstract surfaces, showing that they form a two-parameter family. Additionally, we provide three characterizations of the intrinsic condition to explore the geometric properties of these surfaces.
\end{abstract}



\maketitle

\section{Introduction}
The study of submanifolds has been a central topic in differential geometry, motivated by the necessity to understand the intrinsic and extrinsic properties of manifolds immersed in higher-dimensional spaces. Among the various types of submanifolds, biharmonic and biconservative submanifolds have gained significant attention due to their intriguing properties.

Biharmonic submanifolds generalize the concept of well-known minimal submanifolds and they are isometric immersions $\varphi:\left(M^m,g\right)\to\left(N^n,h\right)$ satisfying the biharmonic equation
$$
\tau_2(\varphi)=-\Delta^\varphi \tau(\varphi)-\trace R^N\left(\varphi_\ast,\tau\left(\varphi\right)\right)\varphi_\ast=0.
$$
Here, $\Delta^\varphi$ is the rough Laplacian acting on sections of the pull-back bundle $\varphi^{-1}\left(TN^n\right)$, $R^N$ is the curvature tensor field on $N^n$, and 
$$
\tau(\varphi)=mH
$$ 
is the tension field associated to $\varphi$, where $H$ is the mean curvature vector field.

The bitension field $\tau_2(\varphi)$ has a tangent and a normal part, and the biconservative submanifolds are characterized by the vanishing of its tangent part. For more details about the geometric meaning of the equation $\tau_2(\varphi)^\top=0$, see for example \cite{CMOP, Jiang87, LMO, NPhD17}.

The investigation of biconservative submanifolds has led to numerous results characterizing these submanifolds in various ambient spaces, including Euclidean spaces, Euclidean spheres, and hyperbolic spaces. Recent research has focused on classifying these submanifolds under specific geometric conditions. 

Naturally, the first step on studying biconservative submanifolds was to investigate the properties of biconservative hypersurfaces in space forms, i.e., spaces with sectional curvature $\varepsilon$, and to classify them, when it is possible. For the first paper with this topic see \cite{Hasanis-Vlachos}. In space forms, the biconservative hypersurfaces are characterized by
$$
A(\grad f)=-\frac{m}{2}f\grad f,
$$
where $A$ is the shape operator and $f=(\trace A)/ m$ is the mean curvature function. The constant mean curvature hypersurfaces, i.e., $CMC$ hypersurfaces, are trivially biconservative, so the interesting case is the study of non-$CMC$ biconservative hypersurfaces. For a recent survey on this topic see \cite{FO2022}.

A next step could be represented by the study of biconservative submanifolds of codimension $2$, and the simplest case is the study of biconservative surfaces in $4$-dimensional space forms $N^4(\varepsilon)$. In this context, the surfaces with parallel mean curvature vector fields, i.e., $PMC$ surfaces, are trivially biconservative, so the interesting case is the study of non-$PMC$ biconservative surfaces. 

The $CMC$ biconservative surfaces in $N^4(\varepsilon)$ were classified in \cite{MOR2016JGA}. The natural following step is to examine non-$CMC$ biconservative surfaces in $N^4(\varepsilon)$. However, classifying these surfaces without additional assumptions seems to appear quite difficult. A particularly helpful condition for our study is to require that the surfaces have parallel normalized mean curvature vector field ($PNMC$). 

The (non-$CMC$) $PNMC$ biconservative surfaces in $N^4(\varepsilon)$, with $\varepsilon=0$ and $\varepsilon=1$ where studied in \cite{YeginTurgay2018} and \cite{NOTYS}, respectively. In this paper, we aim to advance the study of biconservative surfaces in $4$-dimensional space forms by considering $\varepsilon=-1$, i.e., the hyperbolic space, and exploring their fundamental properties, providing new classifications, and offering geometric characterizations that shed light on their intrinsic geometry. 

In our work, we slightly modify the technique used in \cite{NOTYS, YeginTurgay2018}. Inspired by \cite{Hasanis-Vlachos}, we construct the local charts in a more geometric way, using the flow of the unit vector field $(\grad f) /|\grad f|$, where $f$ is the mean curvature function. We also altered the overall approach for finding the family of abstract surfaces that can admit a $PNMC$ biconservative immersion. Furthermore, some intrinsic characterizations of $PNMC$ biconservative surfaces in the hyperbolic space $\mathbb{H}^4$ are presented. A crucial aspect of the intrinsic approach involves considering the Gauss equation as a cubic polynomial equation in $f$ which is dependent on a certain constant $c$. This constant $c$ cannot be seen as a parameter, as it belongs to the intrinsic geometry of the surface.

The paper is organized like follows. In Section \ref{preliminaries} we recall the fundamental equations of submanifolds and introduce some notations that are used throughout the paper. In Section \ref{sec-intrinsic} we begin with some properties of $PNMC$ biconservative surfaces in $\mathbb{H}^4$. These are included in Theorem \ref{thm:fundamentalProperties}. More precisely, here are described the Levi-Civita connection on the domain manifold, the shape operators, the link between the Gaussian curvature $K$ and the mean curvature function $f$, the equation that is satisfied by $f$ and two equivalent expression for the domain metric $g$. The next main result, Proposition \ref{prop-isometry}, states that there exists a two-parametric family of abstract surfaces $\left(M^2,g\right)$ with $g$ given in Theorem \ref{thm:fundamentalProperties}. Further, in Theorem \ref{first-charact-theorem-intrinsic}, we prove that an abstract surface $\left(M^2,g\right)$ can admit a $PNMC$ biconservative immersion in $\mathbb{H}^4$ if and only if $g$ is as described in Theorem \ref{thm:fundamentalProperties}. Additionally, the uniqueness of such immersions was established in Theorem \ref{thm:uniquenessOfImmersions}. Therefore, dermining all $PNMC$ biconservative immersions in $\mathbb{H}^4$ is equivalent to classifying all abstract surfaces $\left(M^2,g\right)$ with $g$ given in Theorem \ref{thm:fundamentalProperties}. Now, even though the metric $g$ has an explicit analytic expression, we believe that offering some geometric characterizations will lead to a deeper understanding of these metrics. One of these characterizations states that such metrics are defined by the property that the level curves of $K$ are circles with a specific curvature, as detailed in Theorem  \ref{second-charact-bicons-surfaces-intrinsic}.  

\textbf{Conventions and notations.} All Riemannian metrics are indicated, in general, by the same symbol $\langle\cdot,\cdot\rangle$. Sometimes, when there is no confusion, we omit to indicate the metric. We assume that the manifolds are connected and oriented and, for the rough Laplacian acting on sections of the pull-back bundle $\varphi^{-1}\left(TN^n\right)$ and for the curvature tensor field, the following sign conventions are used:
$$
\Delta^{\varphi}=-\trace\left(\nabla^{\varphi}\nabla^{\varphi}-\nabla^{\varphi}_{\nabla}\right)
$$
and
$$
R(X,Y)Z=[\nabla_X,\nabla_Y]Z-\nabla_{[X,Y]}Z,
$$
respectively. Here, $\varphi:M^m\to N^n$ is a smooth map between two Riemannian manifolds, $\nabla^\varphi$ denotes the induced connection on  $\varphi^{-1}\left(TN^n\right)$ and  $\nabla$ is the Levi-Civita connection on $M^m$. 

Essentially, our approach is of local nature. In order to avoid trivial cases for our study of non-$CMC$ biconservative surfaces $M^2$ in $\mathbb{H}^4$, we assume that the mean curvature function of the surface is positive, its gradient is different to zero at any point, and $M^2$ is completely contained in $\mathbb{H}^4$, i.e., any open subset of $M^2$ cannot lie in any totally geodesic $\mathbb{H}^3\subset\mathbb{H}^4$.

\section{Preliminaries}\label{preliminaries}
Let $\varphi:\left(M^m,g\right)\to \left(N^n,h\right)$ be an isometric immersion or, simply, let $M^{m}$ be a submanifold in $N^{n}$. We have the standard decomposition of $\varphi^{-1}\left(TN^n\right)$ into the direct sum of the tangent bundle $TM^m\equiv \varphi_\ast \left(TM^m\right)$ of $M^m$ and the normal bundle 
$$
NM^m=\displaystyle{\bigcup_{p\in M} \left(\varphi_\ast \left(T_{p}M^m\right)\right)^\perp}
$$ 
of $M^m$ in $N^n$. 

Since we work locally, $M^m$ is identified with its image by $\varphi$, a vector field $X$ tangent to $M^m$ becomes a vector field tangent to $N^n$ along $\varphi(M)\equiv M$, and $\nabla^\varphi_X \varphi_\ast(Y)$ is now identified with $\nabla^N_X Y$, where $\nabla^N$ is the Levi-Civita connection on $N^n$. Next, we recall the Gauss and the Weingarten formulas
\begin{equation*}
	\nabla^N_X Y=\nabla_X Y+B(X,Y),
\end{equation*}
and
\begin{equation*}
	\nabla^N_X \eta=-A_\eta(X)+\nabla^\perp_X\eta,
\end{equation*}
where $B\in C\left(\odot^2 T^\ast M^m\otimes NM^m\right)$ is called the second fundamental form of $M^m$ in $N^n$, $A_\eta\in C\left(T^\ast M^m\otimes TM^m\right)$ is the shape operator of $M^m$ in $N^n$ in the normal direction $\eta$, and $\nabla^\perp$ is the induced connection in the normal bundle.

In the particular cases when $N^n=\mathbb{H}^4$ or $N^n=\mathbb{R}^5$, we denote the corresponding Levi-Civita connections by $\tilde{\nabla}$ or $\hat{\nabla}$, respectively.

The mean curvature vector field of $M^m$ in $N^n$ is defined by $H=(\trace B)/m\in C\left(NM^m\right)$, where the $\trace$ is considered with respect to the domain metric $g$. The mean curvature function is defined by $f=|H|$.

In this paper, we assume that $H\neq 0$ at any point, so $f$ is a smooth positive function on $M^m$. We denote $E_{m+1}=H/f$ and $A_{m+1}=A_{E_{m+1}}$. Also, a local orthonormal frame field in the normal bundle $NM^m$ of $M^m$ in $N^n$ is indicated by $\left\{E_{m+1},\ldots, E_n\right\}$. 
 
We recall now the fundamental equations of submanifolds, i.e.,  the \textit{Gauss, Codazzi and Ricci equations}, that we will use later in this paper:

\begin{equation}\label{Gauss-equation}
\langle R^N(X,Y)Z,W\rangle=\langle R(X,Y)Z,W\rangle-\langle B(X,W),B(Y,Z)\rangle+\langle B(Y,W),B(X,Z)\rangle,
\end{equation}

\begin{equation}\label{Codazzi-equation}
	\left(\nabla_X A_\eta\right)(Y)-\left(\nabla_Y A_\eta\right)(X)=A_{\nabla_X^\perp \eta}(Y)-A_{\nabla_Y^\perp \eta}(X)-\left(R^N(X,Y)\eta\right)^\top,
\end{equation}
and 
\begin{equation}\label{Ricci-equation}
	\left(R^N(X,Y)\eta\right)^\perp=R^\perp(X,Y)\eta+B\left(A_\eta(X),Y\right)-B\left(X,A_\eta(Y)\right),
\end{equation} 
where $X, Y, Z, W\in C\left(TM^m\right)$ and $\eta\in C\left(NM^m\right)$.

Concerning the biconservative submanifolds, we only mention here some characterization formulas.

\begin{proposition}
Let $\varphi:\left(M^m,g\right)\to \left(N^n,h\right)$ be a submanifold. Then, the following conditions are equivalent:
	
\begin{itemize}
\item [(i)] $M^m$ is biconservative;
\item [(ii)] $\trace A_{\nabla^\perp_{\cdot} H}(\cdot)+\trace \nabla A_H +\trace \left( R^N(\cdot,H)\cdot\right)^\top=0$;
\item [(iii)] $\frac{m}{2}\grad\left(|H|^2\right)+2\trace A_{\nabla^\perp_{\cdot} H}(\cdot) + 2\trace \left(R^N(\cdot,H)\cdot\right)^\top=0$;
\item [(iv)] $2\trace \nabla A_H-\frac{m}{2}\grad\left(|H|^2\right)=0$.
	
\end{itemize}	
\end{proposition}

\section{Intrinsic characterization of $PNMC$ biconservative surfaces}\label{sec-intrinsic}

We study $PNMC$ biconservative surfaces $\varphi:\left(M^2,g\right)\to\mathbb{H}^4$. We recall (see \cite{FLO2021}) that $M^2$ is a $PNMC$ biconservative surface in $\mathbb{H}^4$, i.e.,  $\tau_2^\top(\varphi)=0$, if and only if 
\begin{equation}\label{eq:PNMC-biconservative}
A_{3}(\grad f)=-f\grad f.
\end{equation}

Let us consider
$$
E_1=\frac{\grad f}{|\grad f|} \qquad \text{and} \qquad E_3=\frac{H}{f}.
$$

Because of orientation, we can consider the positively oriented global orthonormal frame fields $\left\{E_1,E_2\right\}$ in the tangent bundle $TM^2$ and $\left\{E_3,E_4\right\}$ in the normal bundle $NM^2$.

Clearly, $E_2f=0$. Denoting $A_4=A_{E_4}$, we have the following intrinsic and extrinsic properties of our surfaces, which mainly follow from \eqref{Gauss-equation}, \eqref{Codazzi-equation} and \eqref{Ricci-equation}.

\begin{theorem}\label{thm:fundamentalProperties}
Let $\varphi:\left(M^2,g\right)\to\mathbb{H}^4$ be a $PNMC$ biconservative surface. Then, the following hold:
\begin{itemize}
\item [(i)] the Levi-Civita connection $\nabla$ of $M^2$ and the normal connection $\nabla^\perp$ of $M^2$ in $\mathbb{H}^4$ are given by
\begin{equation}\label{Levi-Civita-connection-f}
\nabla_{E_1}E_1=\nabla_{E_1}E_2=0, \quad \nabla_{E_2}E_1=-\frac{3}{4}\frac{E_1 f}{f}E_2, \quad \nabla_{E_2}E_2=\frac{3}{4}\frac{E_1 f}{f}E_1
\end{equation}
and
\begin{equation*}\label{normal-connection-f}
\nabla^\perp E_3=0, \qquad \nabla^\perp E_4=0;
\end{equation*}

\item [(ii)] the shape operators corresponding to $E_3$ and $E_4$ are given, with respect to $\left\{E_1, E_2\right\}$, by the matrices
\begin{equation*}\label{shape-operators-A3-A4-f}
A_3=\left(
\begin{array}{cc}
-f & 0 \\
 0 & 3f
\end{array}
\right),\quad
A_4=\left(
\begin{array}{cc}
cf^{3/2} & 0 \\
0 & -cf^{3/2}
\end{array}
\right),
\end{equation*}
where $c$ is a non-zero real constant;

\item [(iii)] the Gaussian curvature $K$ and the mean curvature function $f$ are related by
\begin{equation}\label{relation-K-f}
K=-1-3f^2-c^2f^3,
\end{equation}
thus $1+K<0$ on $M^2$;

\item [(iv)] the mean curvature function $f$ satisfies
\begin{equation}\label{second-order-invariant-f}
f\Delta f+\left|\grad f\right|^2-\frac{4}{3}f^2-4f^4+\frac{4}{3}c^2f^5=0;
\end{equation}

\item [(v)] around any point of $M^2$ there exists a positively oriented local chart $X^f=X^f(u,v)$ such that
$$
\left(f\circ X^f\right)(u,v)=f(u,v)=f(u)
$$
and $f$ satisfies the following second order $ODE$
\begin{equation}\label{second-order-chart-f}
 f''f-\frac{7}{4}\left(f'\right)^2+\frac{4}{3}f^2+4f^4+\frac{4}{3}c^2f^5=0
\end{equation}
and the condition $f'>0$. The first integral of the above second order ODE is
\begin{equation}\label{first-integral-f}
\left(f'\right)^2-\frac{16}{9}f^2+16f^4+\frac{16}{9}c^2f^5-2Cf^{7/2}=0,
\end{equation}
where $C$ is a real constant.

Moreover, the metric $g$ is given by 
\begin{equation}\label{expression-metricg-uv}
g(u,v)=du^2+\frac{1}{f^{3/2}(u)}dv^2.
\end{equation}

\item [(vi)] around any point of $M^2$ there exist positively oriented local coordinates $(f,v)$ such that the metric $g$ can be written as
\begin{equation}\label{g(f,v)}
g(f,v)=\frac{1}{\frac{16}{9}f^2-16f^4-\frac{16}{9}c^2f^5+2Cf^{7/2}}df^2+\frac{1}{f^{3/2}}dv^2,
\end{equation}
with $C\in\mathbb{R}$ and $c\neq 0$.

\end{itemize}
\end{theorem}

\begin{proof}
First, using \eqref{eq:PNMC-biconservative}, we have $A_3\left(E_1\right)=-fE_1$ and then, since $\trace A_3=2f$, we get $A_3\left(E_2\right)=3fE_2$.

As $\nabla^\perp E_3=0$, we have
$$\nabla^\perp E_4=0, \quad R^\perp(X,Y)E_3=0 \quad \text{and} \quad R^\perp(X,Y)E_4=0,
$$
for any $X,Y\in C\left(TM^2\right)$.

From the Ricci equation we obtain $B\left(E_1,E_2\right)=0$, so $\langle A_4\left(E_1\right), E_2\rangle=0$. On the other hand, since $\trace A_4=0$, we obtain that the matrix of $A_4$ with respect to $\left\{E_1, E_2\right\}$ is given by
\begin{equation*}
A_4=\left(
\begin{array}{cc}
\lambda & 0 \\
0 & -\lambda
\end{array}
\right),
\end{equation*}
for some smooth function $\lambda$ on $M^2$.

If we assume that $\lambda=0$ on $M^2$, or on an arbitrary open subset of $M^2$, we obtain that $\hat{\nabla}_X E_4=0$ for any $X\in C(T\mathbb{H}^4)$, hence $E_4$ is a constant vector field. Then, for any point $p\in M^2$, we can identify $(i\circ \varphi)(p)$ with its position vector in $\mathbb{R}^5$, where $i:\mathbb{H}^4\to \mathbb{R}^5$ is the canonical inclusion, and we have $(i\circ \varphi)(p)\perp E_4$. It follows that $(i\circ \varphi)(M)$ belongs to the hyperplane that passes through the origin and whose normal is $E_4$. Next, by standard arguments from the hyperbolic geometry (the hyperboloid model), we conclude that $\varphi(M)\subset\mathbb{H}^3\subset\mathbb{H}^4$ and we get a contradiction. 

Therefore, $\lambda\neq 0$ at any point of an open dense subset of $M^2$. For simplicity, we will assume that $\lambda\neq 0$ at any point of $M^2$.

In order to obtain a more explicit expression of $\lambda$, we use the Codazzi equation. More precisely, if we consider \eqref{Codazzi-equation} applied for $\eta=E_3$ we obtain the connection forms
\begin{equation}\label{connectionForms}
\omega^1_2\left(E_1\right)=0, \qquad \omega^1_2\left(E_2\right)= \frac{3}{4}\frac{E_1f}{f}.
\end{equation}
Then, from \eqref{Codazzi-equation} applied for $\eta=E_4$, we get
$$
E_1\lambda=\frac{3\lambda}{2}\frac{E_1 f}{f}, \qquad E_2\lambda=0.
$$
Thus,
$$
\frac{E_1\lambda}{\lambda}=\frac{3}{2}\frac{E_1f}{f},
$$
which is equivalent to
$$
E_1\left(\ln \frac{|\lambda|}{f^{3/2}}\right)=0.
$$
Moreover, as $E_2\lambda=0$ and $E_2f=0$, it follows that the function $\ln\left(|\lambda|/f^{3/2}\right)$ is constant, and therefore
$$
\lambda=cf^{3/2},
$$
where $c$ is a non-zero real constant.

We note that even if we had worked on a connected subset of the set of all points where $\lambda\neq 0$,  we would have obtained (from the expression of $A_4$ and the fact that $f>0$ everywhere on $M^2$) that the constant $c$ does not depend on that connected component.

Further, using \eqref{connectionForms}, it is easy to see that the Levi-Civita connection of $M^2$ is given by \eqref{Levi-Civita-connection-f}.

From the expressions of $A_3$ and $A_4$, we have
$$
B\left(E_1,E_1\right)=-fE_3+cf^{3/2}E_4, \qquad B\left(E_2,E_2\right)=3fE_3-cf^{3/2}E_4,
$$
and, thus, applying the Gauss equation, we obtain the relation between $K$ and $f$ as follows
$$
K=-1-3f^2-c^2f^3.
$$
Next, we want to obtain \eqref{second-order-invariant-f}. In order to do this, we recall that
$$
d\omega^1_2\left(E_1,E_2\right)=K\left(\omega^1\wedge\omega^2\right)\left(E_1,E_2\right),
$$
i.e,
$$
E_1\left(\omega^1_2\left(E_2\right)\right)-E_2\left(\omega^1_2\left(E_1\right)\right)-\omega^1_2\left(\left[E_1,E_2\right]\right)=K.
$$
Using \eqref{Levi-Civita-connection-f} and \eqref{connectionForms}, the above relation can be rewritten as
$$
K=\frac{3}{16f^2}\left[4E_1\left(E_1 f\right)f-7\left(E_1f\right)^2\right].
$$
Now, replacing $K$ from \eqref{relation-K-f}, we get
\begin{equation}\label{secondDerivativefInvariant}
f E_1\left(E_1f\right) =\frac{7}{4}\left(E_1f\right)^2-\frac{4}{3}f^2-4f^4-\frac{4}{3}c^2f^5.
\end{equation}
From the expression of the Laplacian of $f$, we obtain
\begin{equation*}
\Delta f =-E_1\left(E_1f\right)+\frac{3}{4}\frac{\left(E_1f\right)^2}{f},
\end{equation*}
and therefore, \eqref{secondDerivativefInvariant} is equivalent to \eqref{second-order-invariant-f}. Consequently, we get (iv).

Further, we construct around any point of $M^2$ a positively oriented local chart $X^f=X^f(u,v)$ such that
$$
\left(f\circ X^f\right)(u,v)=f(u,v)=f(u),
$$
equation \eqref{second-order-invariant-f} to be equivalent to a second order $ODE$ and the metric $g$ to have the expression given in \eqref{expression-metricg-uv}. In order to obtain such a local chart, first we construct a geometric one, involving the flow of $E_1$ which is made up by geodesics of $\left(M^2,g\right)$, and then we perform a simple change of coordinates.

Let $p_0\in M^2$ be an arbitrarily fixed point of $M^2$ and let $\sigma=\sigma\left(v_1\right)$ be an integral curve of $E_2$ with $\sigma(0)=p_0$. Let $\left\{\phi_{u_1}\right\}_{u_1\in\mathbb{R}}$ be the flow of $E_1$. We define the following positively oriented local chart
$$
X^f\left(u_1,v_1\right)=\phi_{u_1}\left(\sigma\left(v_1\right)\right)=\phi_{\sigma\left(v_1\right)}\left(u_1\right).
$$
We have 
\begin{align*}
&X^f\left(0,v_1\right)=\sigma\left(v_1\right),\\
&X^f_{v_1}\left(0,v_1\right)=\sigma'\left(v_1\right)=E_2\left(\sigma\left(v_1\right)\right)=E_2\left(0,v_1\right), \\
&X^f_{u_1}\left(u_1,v_1\right)=\phi'_{\sigma\left(v_1\right)}\left(u_1\right)=E_1\left(\phi_{\sigma\left(v_1\right)}\left(u_1\right)\right)=E_1\left(u_1,v_1\right),
\end{align*}
for any $\left(u_1,v_1\right)$. Clearly, $\left\{X^f_{u_1},X^f_{v_1}\right\}$ is positively oriented and as $E_1f>0$, we also have $X_{u_1}^f f>0$.

If we write the Riemannian metric $g$ on $M^2$ in local coordinates as
$$
g=\langle \cdot, \cdot \rangle=g_{11}\ du_1^2+2g_{12}\ du_1\  dv_1+g_{22}\  dv_1^2,
$$
we get $g_{11}=\left|X^f_{u_1}\right|^2=\left|E_1\right|^2=1$, $g_{22}\left(0,v_1\right)=1$ and $g_{12}\left(0,v_1\right)=0$. Then, it is not difficult to see that 
\begin{equation}\label{expression-E_2}
E_2=\frac{1}{\sqrt{\det G}}\left(-g_{12}X^f_{u_1}+\ X^f_{v_1}\right),
\end{equation}
where
\begin{equation*}
	G=G\left(u_1,v_1\right)=\left(
	\begin{array}{cc}
		g_{11} & g_{12} \\
		g_{12}  & g_{22}
	\end{array}
	\right).
\end{equation*}
Clearly,  $\det G=g_{22}-g_{12}^2>0$ and $\det G\left(0,v_1\right)=1$.

Further, from $E_2f=0$ and from the Levi-Civita connection of $M^2$ we have $\left[E_1,E_2\right]f=-E_2\left(E_1f\right)=0$. Now, using $E_1= X^f_{u_1}$ and \eqref{expression-E_2}, by some standard computations we achieve that $E_2\left(E_1f\right)=0$ is equivalent to $X^f_{u_1} g_{12}=0$, i.e., 
$g_{12}\left(u_1,v_1\right)=g_{12}\left(0,v_1\right)=0$.  Therefore, we achieve 
$$
E_1=X^f_{u_1}=\grad u_1,\qquad E_2=\frac{1}{\sqrt{g_{22}}} X_{v_1}^f.
$$
Since $E_1f>0$ and $E_2f=0$ it follows that $X_{u_1}^f f>0$ and $X^f_{v_1}\ f=0$. Thus,
$$
f\left(u_1,v_1\right)=f\left(u_1,0\right)=f\left(u_1\right),
$$
i.e., the level curves of $f$ are given by $u_1=const$, and $f'\left(u_1\right)>0$. From the expression of the Levi-Civita connection, we obtain
\begin{align*}
	\nabla_{E_2}{E_1}& = -\frac{3}{4}\frac{f'}{f}\frac{1}{\sqrt{g_{22}}} X_{v_1}^f\\
	& = \frac{1}{\sqrt{g_{22}}} \left(\Gamma_{21}^1 X_{u_1}^f+\Gamma_{21}^2X_{v_1}^f\right),
\end{align*}
where $\Gamma_{ij}^k$ represent the Christoffel symbols. Therefore $\Gamma_{21}^1=0$ and 
\begin{equation}\label{eq1-chiristoffel}
\Gamma_{21}^2=-\frac{3}{4}\frac{f'}{f}.
\end{equation}
On the other hand, we know that
\begin{align}\label{eq2-chirstoffel}
\Gamma_{21}^2 & =\frac{1}{2}g^{22}\left(X^f_{v_1}g_{12}+X^f_{u_1} g_{22}-X^f_{v_1} g_{21}\right) \nonumber\\
& = \frac{1}{2}\frac{X^f_{u_1} g_{22}}{g_{22}}.
\end{align}
From \eqref{eq1-chiristoffel} and $\eqref{eq2-chirstoffel}$, by an integrating process we get
$$
g_{22}\left(u_1,v_1\right)=\frac{\alpha\left(v_1\right)}{f^{3/2}\left(u_1\right)},
$$
where $\alpha$ is a positive smooth function.

Thus, the metric $g$ can be written as
$$
g\left(u_1,v_1\right)=du_1^2+\frac{\alpha\left(v_1\right)}{f^{3/2}\left(u_1\right)}dv_1^2
$$
and 
$$
E_2=E_2\left(u_1,v_1\right)=\frac{f^{3/4}\left(u_1\right)}{\sqrt{\alpha\left(v_1\right)}}X^f_{v_1}.
$$
Further, if we consider the change of coordinates 
$$
\left(u_1,v_1\right) \to \left(u=u_1,v=\int_{0}^{v_1} \sqrt{\alpha(\tau)}\ d\tau \right),
$$
we achieve $f=f\left(u_1\right)=f(u)$, $f'(u)>0$,
$$
g\left(u,v\right)=du^2+\frac{1}{f^{3/2}\left(u\right)}dv^2,
$$
\begin{equation}\label{E1E2-second chart}
E_1=\frac{\partial}{\partial u} \qquad\text{ and }\qquad E_2=f^{3/4}(u)\frac{\partial}{\partial v}.
\end{equation}
Replacing $E_1$ and $E_2$ from \eqref{E1E2-second chart} in \eqref{secondDerivativefInvariant}, by some straightforward computations we get \eqref{second-order-chart-f}.

Next, we find the first integral of \eqref{second-order-chart-f}. First, we note that for any smooth function we have
\begin{equation}\label{eq:helpfullIntegralPrime}
	\frac{1}{2}\left(\frac{\left(f'\right)^2}{f^{7/2}}\right)'=\frac{f'f''}{f^{7/2}}-\frac{7}{4}\frac{\left(f'\right)^3}{f^{9/2}}.
\end{equation}
Then, we multiply \eqref{second-order-chart-f} by $f'/f^{9/2}$ and using \eqref{eq:helpfullIntegralPrime} we obtain
$$
\left(\frac{1}{2}\frac{\left(f'\right)^2}{f^{7/2}}+\frac{8}{9}c^2f^{3/2}+8\sqrt{f}-\frac{8}{9f^{3/2}}\right)'=0.
$$
Integrating the above relation one obtains \eqref{first-integral-f}. 

In the following, we change one more time the local coordinates. Indeed, if we consider the change of coordinates
$$
(u,v)\to (f=f(u),v),
$$ 
using \eqref{first-integral-f}, we find a very explicit expression of the domain metric $g$ as
$$
g(f,v)=\frac{1}{\frac{16}{9}f^2-16f^4-\frac{16}{9}c^2f^5+2Cf^{7/2}}df^2+\frac{1}{f^{3/2}}dv^2.
$$
Moreover, the vector fields $E_1$ and $E_2$ can be rewritten as
$$
E_1=\sqrt{\frac{16}{9}f^2-16f^4-\frac{16}{9}c^2f^5+2Cf^{7/2}}\frac{\partial}{\partial f} \qquad\text{ and }\qquad E_2=f^{3/4}\frac{\partial}{\partial v},
$$
with $c\neq 0$, $C\in \mathbb{R}$.
\end{proof}

\begin{remark}\label{remark-equivalence}
	In the above proof we have seen that the hypothesis $M^2$ is completely contained in $\mathbb{H}^4$ implies that the rank of the first normal bundle is equal to $2$. In fact, it is not difficult to check that the converse holds too. Also, from the expression of the shape operator $A_3$, we note that $M^2$ cannot be pseudo-umbilical at any point.
\end{remark}

\begin{remark}
Equations \eqref{second-order-chart-f} and \eqref{first-integral-f} that are satisfied by $f$ are invariant under changes of argument of type $u\to \pm\tilde{u}+const$.
\end{remark}

\begin{remark}
From the above proof we can see that the integral curves of $E_2$ are circles.
\end{remark}

If we want to classify, up to intrinsic isometries, the Riemannian metrics given by \eqref{expression-metricg-uv} it is easier to work for this issue with their equivalent form  \eqref{g(f,v)}. Let $\left(M^2_1,g_1\right)$ and $\left(M^2_2,g_2\right)$ be two abstract surfaces given by
$$
g_i\left(f_i,v_i\right)=\theta_i\left(f_i\right)df_i^2+\frac{1}{f_i^{3/2}}dv_i^2, \qquad i\in\left\{1,2\right\} ,
$$
where 
$$
\theta_i\left(f_i\right)=\frac{1}{\frac{16}{9}f_i^2-16f_i^4-\frac{16}{9}c_i^2f_i^5+2C_if_i^{7/2}}>0,
$$
$C_i$ and $c_i$ are some real constants, $c_i\neq 0$.

\begin{proposition}\label{prop-isometry}
	If there exists an intrinsic isometry $\Psi:\left(M^2_1,g_1\right)\to\left(M^2_2,g_2\right)$, then	
	$$
	\Psi\left(f_1,v_1\right)=\left(f_1,\pm v_1+b\right),
	$$ 
	where $b$ is some real constant, $c_1^2=c_2^2$ and $C_1=C_2$. In particular, there exists a two-parametric family of Riemannian metric $g$ given by \eqref{g(f,v)}.
\end{proposition}

\begin{proof}
	Let us consider 
	$$
	\Psi\left(f_1,v_1\right)=\left(\Psi^1\left(f_1,v_1\right), \Psi^2\left(f_1,v_1\right)\right)	
	$$
	an isometry between $\left(M^2_1,g_1\right)$ and $\left(M^2_2,g_2\right)$, i.e., $\Psi^\ast g_2=g_1$. Thus, the following relations hold
	\begin{equation}\label{eq1-isometry}
		\left(\frac{\partial \Psi^1}{\partial f_1}\right)^2\left(f_1,v_1\right)\cdot\theta_2\left(\Psi^1\left(u_1,v_1\right)\right)+\left(\frac{\partial \Psi^2}{\partial f_1}\right)^2\left(f_1,v_1\right)\cdot\frac{1}{\left(\Psi^1\left(f_1,v_1\right)\right)^{3/2}}=\frac{1}{f_1^{3/2}},
	\end{equation}
	\begin{equation}\label{eq2-isometry}
		\frac{\partial \Psi^1}{\partial f_1}\left(f_1,v_1\right) \cdot\frac{\partial \Psi^1}{\partial v_1}\left(f_1,v_1\right)\cdot\theta_2\left(\Psi^1\left(u_1,v_1\right)\right)+\frac{\partial \Psi^2}{\partial f_1}\left(f_1,v_1\right)\cdot \frac{\partial \Psi^2}{\partial v_1}\left(f_1,v_1\right)\cdot\frac{1}{\left(\Psi^1\left(f_1,v_1\right)\right)^{3/2}}=0
	\end{equation}
	and
	\begin{equation}\label{eq3-isometry}
		\left(\frac{\partial \Psi^1}{\partial v_1}\right)^2\left(f_1,v_1\right)\cdot\theta_2\left(\Psi^1\left(u_1,v_1\right)\right)+\left(\frac{\partial \Psi^2}{\partial v_1}\right)^2\left(f_1,v_1\right)\cdot\frac{1}{\left(\Psi^1\left(f_1,v_1\right)\right)^{3/2}}=\frac{1}{f_1^{3/2}},
	\end{equation}
	for any $f_1$ and $v_1$.
	Moreover, we also know that
	\begin{equation}\label{eq4-isometry}
		K_2\left(\Psi^1\left(f_1,v_1 \right)\right)=K_1\left(f_1\right),
	\end{equation}
   where $K_i$ represents the Gaussian curvature of $M_i^2$. If we take the derivative of \eqref{eq4-isometry} with respect to $v_1$, since $K_2'\neq 0$, we easily obtain that $\Psi^1=\Psi^1\left(f_1\right)$. Further, knowing that the function $\Psi^1$ depends only on $f_1$, from \eqref{eq2-isometry}, we get
	\begin{equation}\label{isom1}
		\frac{\partial \Psi^2}{\partial f_1}\left(f_1,v_1\right) \frac{\partial \Psi^2}{\partial v_1}\left(f_1,v_1\right)=0,
	\end{equation}
	and, from \eqref{eq3-isometry}, we have
	\begin{equation}\label{isom2}
		\left(\frac{\partial \Psi^2}{\partial v_1}\right)^2\left(f_1,v_1\right)\cdot\theta_2\left(\Psi^1\left(u_1\right)\right)=\frac{1}{f_1^{3/2}}\neq 0.
	\end{equation}
	Now, from \eqref{isom1} and \eqref{isom2}, one obtains that $\Psi^2=\Psi^2\left(v_1\right)$, and then, from \eqref{eq3-isometry} we get
	$$
	\left(\frac{d \Psi^2}{d v_1}\right)^2\left(v_1\right)=\left(\frac{\Psi^1\left(f_1\right)}{f_1}\right)^{3/2}.
	$$
	Since the left hand side of the above relation depends only of $v_1$ and the right hand side depends only of $f_1$, it follows that
	$$
	\Psi^1\left(f_1\right)=af_1, \qquad \Psi^2\left(v_1\right)=\pm a^{3/4}v_1+b,
	$$
	for any $f_1$ and $v_1$, where $a$ and $b$ are some real constants, $a>0$.	
	
	Next, from \eqref{eq1-isometry}, after some standard computations, since $f_1>0$, we obtain
	$$
	\frac{8}{9}\left(c_1^2-c_2^2a^3\right)f_1^{3/2}+8\left(1-a^2\right)\sqrt{f_1}-\left(C_1-C_2a^{3/2}\right)=0,
	$$
	for any $f_1$.
	Therefore, we get
	$$
	a=1, \qquad c_1^2=c_2^2, \qquad C_1=C_2,
	$$
	and 
	$$
	\Psi\left(f_1,v_1\right)=\left(f_1, \pm v_1+b\right), \qquad b\in\mathbb{R}.
	$$
\end{proof}

Further, in order to obtain more properties of the abstract surface $\left(M^2,g\right)$ which admits a $PNMC$ biconservative immersion in $\mathbb{H}^4$, we prefer to work with the expression \eqref{expression-metricg-uv} of the domain metric $g$ and not with the very explicit one \eqref{g(f,v)}. 

The next result shows that, given an abstract surface $\left(M^2,g\right)$, if it admits a $PNMC$ biconservative immersion $\varphi$ in $\mathbb{H}^4$, then it is unique. In particular, it follows that the function $f$ is unique and, up to the sign, the constant $c$ is unique too, depends on $\left(M^2,g\right)$ and it is not an indexing constant. This phenomenon is different from the case of minimal surfaces in a $3$-dimensional space form where, given an abstract surface $\left(M^2,g\right)$ satisfying a certain intrinsic condition, there exists a one-parametric family of minimal immersions (see \cite{L70,MM2015,R95}). Because the proof of our result closely resembles that of Theorem 3.2 in \cite{NOTYS}, we skip it.  

\begin{theorem}\label{thm:uniquenessOfImmersions}
If an abstract surface $\left(M^2,g\right)$ admits two $PNMC$ biconservative immersions in $\mathbb{H}^4$, then these immersions differ by an isometry of $\mathbb{H}^4$.
\end{theorem}

Further, we give a converse of item (v) in Theorem \ref{thm:fundamentalProperties}. More precisely, we have the next result where, in order to directly apply the Fundamental Theorem of Submanifolds in space forms we impose $M^2$ to be simply connected.

\begin{theorem} \label{abstract-to-existence-immersion}
Let $\left(M^2,g\right)$ be an abstract surface. Assume that $M^2$ is simply connected and there exist $(u,v)$ global coordinates such that
$$
g(u,v)=du^2+\frac{1}{f^{3/2}(u)} dv^2,
$$
where $f=f(u)$ is some positive solution of the second order ODE \eqref{second-order-chart-f} with some non-zero real constant $c$, satisfying $f'(u)>0$. Then, there exists a (unique) $PNMC$ biconservative immersion $\varphi:\left(M^2,g\right)\to\mathbb{H}^4$ such that $f$ is its mean curvature.
\end{theorem} 

\begin{proof}
From the expression of the metric $g$, we can compute the Christoffel symbols and then the Gaussian curvature of $M$. More precisely, we obtain
	\begin{align*}
		K & = -\frac{1}{g_{11}}\left(\left(\Gamma_{12}^2\right)_u-\left(\Gamma_{11}^2\right)_v+\Gamma_{12}^1\Gamma_{11}^2+\Gamma_{12}^2\Gamma_{12}^2-\Gamma_{11}^2\Gamma_{22}^2-\Gamma_{11}^1\Gamma_{12}^2\right) \\
		& = \frac{12ff''-21\left(f'\right)^2}{16f^2}.
	\end{align*}

Next, using \eqref{second-order-chart-f} it follows that 
$$
K=-1-3f^2-c^2f^3.
$$ 
Therefore, $1+K<0$, $|\grad f|=f'>0$ and $|\grad K|=-K'>0$ at any point of $M^2$. We set
$$
\tilde{E}_1=\frac{\grad K}{|\grad K|}=-\frac{\grad f}{|\grad f|}=-\frac{\partial}{\partial u}.
$$
and $\tilde{E}_2\in C\left(TM^2\right)$ such that $\left\{\tilde{E}_1,\tilde{E}_2\right\}$ is a positively oriented global orthonormal frame field. It is easy to check that 
$$
\tilde{E}_2=-f^{3/4}(u)\frac{\partial}{\partial v}.
$$
and the Levi-Civita connection of $M^2$ is given by
$$
\nabla_{\tilde{E}_1}\tilde{E}_2=0, \quad \nabla_{\tilde{E}_2}\tilde{E}_1=\frac{3}{4}\frac{f'}{f}\tilde{E}_2, \quad \nabla_{\tilde{E}_2}\tilde{E}_2=-\frac{3}{4}\frac{f'}{f}\tilde{E}_1.
$$
Further, we consider $\Upsilon=M^2\times\mathbb{E}^2$ the trivial vector bundle of rank two over $M^2$. We define $\sigma_3$ and $\sigma_4$  by

\begin{align*}
\sigma_3(p) &=(p,(1,0)), \qquad p\in M^2, \\
\sigma_4(p) &=(p,(0,1)), \qquad p\in M^2,
\end{align*}
which form the canonical global frame field of $\Upsilon$, by $\mathfrak{g}$ the metric on $\Upsilon$ defined by

$$
\mathfrak{g}\left(\sigma_\alpha,\sigma_\beta\right)=\langle \sigma_\alpha,\sigma_\beta \rangle=\delta_{\alpha\beta}, \qquad \alpha,\beta=3,4,
$$
and by $\nabla^\mathfrak{g}$ the connection on $\Upsilon$ given by
$$
\nabla^\mathfrak{g}_{\tilde{E}_i}\sigma_\alpha=0, \qquad i=1,2,\  \alpha=3,4.
$$
Clearly, the pair $\left(\nabla^\mathfrak{g},\mathfrak{g}\right)$ is a Riemannian structure, i.e.,
$$
X\langle\sigma,\rho\rangle=\langle\nabla^\mathfrak{g}_X\sigma,\rho\rangle+\langle \sigma, \nabla^\mathfrak{g}_X\rho\rangle, \qquad X\in C\left(TM^2\right), \ \rho,\sigma\in C\left(\Upsilon\right),
$$
and the curvature tensor field is given by
$$
R^a\mathfrak{g}\left(\tilde{E}_i,\tilde{E}_j\right)\sigma_\alpha=0, \qquad i=1,2,\  \alpha=3,4.
$$
Let us define $B^\mathfrak{g}:C\left(TM^2\right)\times C\left(TM^2\right)\to C\left(\Upsilon\right)$ by
\begin{equation*}
		\left\{
		\begin{array}{lll}
			B^\mathfrak{g}a\left(\tilde{E}_1,\tilde{E}_1\right)=-f\sigma_3+cf^{3/2}\sigma_4 \\\\
			B^\mathfrak{g}\left(\tilde{E}_1,\tilde{E}_2\right)=B^\mathfrak{g}\left(E_2,E_1\right)=0 \\\\
			B^\mathfrak{g}\left(\tilde{E}_2,\tilde{E}_2\right)= 3f\sigma_3-cf^{3/2}\sigma_4
		\end{array}
		\right..
\end{equation*}
Consider $A^\mathfrak{g}_\alpha\in C\left(End\left(TM^2\right)\right)$, $\alpha=3,4$, given by
$$
	\langle A^a_\mathfrak{g}\alpha\left(E_i\right),E_j\rangle=\langle B^\mathfrak{g}\left(E_i,E_j\right),\sigma_\alpha\rangle, \qquad i,j=1,2, \ \alpha=3,4.
$$
We can see that $A^\mathfrak{g}_\alpha$ satisfies, formally, the Gauss, Codazzi and Ricci equations for surfaces in $\mathbb{H}^4$. Therefore, according to the Fundamental Theorem of Submanifolds, locally, there exists an isometric embedding $\varphi:\left(M^2,g\right)\to \mathbb{H}^4$ and a vector bundle isometry $\psi:\Upsilon\to NM$ such that
$$
	\nabla^\perp \psi=\psi\nabla^\mathfrak{g}\qquad \text{and} \qquad B=\psi\circ B^\mathfrak{g},
$$
i.e.,

\begin{equation*}
		\left\{
		\begin{array}{lll}
			\nabla^\perp_{\tilde{E}_i}\left(\psi\left(\sigma_\alpha\right)\right)=\psi\left(\nabla^\mathfrak{g}_{\tilde{E}_i}\sigma_\alpha\right)=0\\\\
			B\left(\tilde{E}_i,\tilde{E}_j\right)=\psi\left(B^\mathfrak{g}\left(\tilde{E}_i,\tilde{E}_j\right)\right), \qquad i,j=1,2,\ \alpha=3,4.
		\end{array}
		\right.
\end{equation*}

As the last step, we denote by $\tilde{E}_\alpha=\psi\left(\sigma_\alpha\right)$ , $\alpha=3,4$. With all the above notations we can see that, with respect to $\left\{\tilde{E}_1, \tilde{E}_2\right\}$, the shape operators $A_{\tilde{E}_\alpha}=A_\alpha$ have the matrices
\begin{equation*}
		A_3=\left(
		\begin{array}{cc}
			-f & 0 \\
			0 & 3f
		\end{array}
		\right),\quad
		A_4=\left(
		\begin{array}{cc}
			cf^{3/2} & 0 \\
			0 & -cf^{3/2}
		\end{array}
		\right),
\end{equation*}
the mean curvature function of the immersion $\varphi$ is $f$ and $\tilde{E}_3=H/f$.
	
Now, we prove that $\varphi$ is a $PNMC$ biconservative immersion in $\mathbb{H}^4$ with $f>0$, $\grad f\neq 0$ at any point, and the surface is completely contained in $\mathbb{H}^4$.
	
First, by straightforward computations, we obtain that $\varphi$ is a $PNMC$ biconservative immersion in $\mathbb{H}^4$ with $f>0$ and $\grad f\neq 0$ at any point. It remains to prove that any open subset of $M^2$ cannot lie in any totally geodesic $\mathbb{H}^3\subset\mathbb{H}^4$. Indeed, since
$$
B\left(\tilde{E}_1,\tilde{E}_1\right)=-f\tilde{E}_3+cf^{3/2}\tilde{E}_4
$$
and
$$
B\left(\tilde{E}_2,\tilde{E}_2\right)= 3f\tilde{E}_3-cf^{3/2}\tilde{E}_4,
$$
it follows that $\left\{B_p\left(\tilde{E}_1,\tilde{E}_1\right), B_p\left(\tilde{E}_2,\tilde{E}_2\right)\right\}$ is a basis in $N_pM^2$, for any $p\in M^2$.
	
Therefore, the first normal bundle $N_1$ defined by
\begin{align*}
		N_1=\Span \Imag(B) & = \Span\left\{B\left(\tilde{E}_1,\tilde{E}_1\right), B\left(\tilde{E}_2,\tilde{E}_2\right)\right\} \\
		& = \Span\left\{\tilde{E}_3,\tilde{E}_4\right\}.
\end{align*}
has the rank equal to $2$ everywhere.
\end{proof}

From Theorems \ref{thm:fundamentalProperties}, \ref{thm:uniquenessOfImmersions} and \ref{abstract-to-existence-immersion} we conclude with the following intrinsic characterization of $PNMC$ biconservative surfaces in $\mathbb{H}^4$.

\begin{theorem}\label{first-charact-theorem-intrinsic}
Let $\left(M^2,g\right)$ be an abstract surface. Then, $M^2$ can be locally uniquely isometrically embedded in $\mathbb{H}^4$ as a $PNMC$ biconservative surface, completely contained in $\mathbb{H}^4$, with positive mean curvature and nowhere vanishing  gradient of the mean curvature, if and only if the metric $g$ is given by
$$
g(u,v)=du^2+\frac{1}{f^{3/2}(u)} dv^2,
$$
where $f=f(u)$ is some positive solution of the second order ODE \eqref{second-order-chart-f} with some non-zero real constant $c$, satisfying $f'(u)>0$. Moreover, in this case, $f$ represents the mean curvature of the $PNMC$ biconservative immersion.
\end{theorem}

From Proposition \ref{prop-isometry} and Theorem \ref{first-charact-theorem-intrinsic} it follows the next corollary.

\begin{corollary}
There exists a two-parametric family of $PNMC$ biconservative surfaces in $\mathbb{H}^4$ with $f>0$ and $\grad f\neq 0$ everywhere. 
\end{corollary}

Further, we give three intrinsic geometric characterizations of the abstract surfaces $\left(M^2,g\right)$ with $g$ given by \eqref{expression-metricg-uv}.

The first one is related to the fact that the level curves of the Gaussian curvature $K$ of $M^2$ are circles. In order to obtain this characterization, we first give 

\begin{theorem} \label{first-intrinsic-charact}
Let $\left(M^2,g\right)$ be an abstract surface. Assume that around any point of $M^2$ there exist $(u,v)$ positively oriented local coordinates such that
$$
g(u,v)=du^2+\frac{1}{f^{3/2}(u)} dv^2,
$$
where $f=f(u)$ is some positive solution of the second order ODE \eqref{second-order-chart-f} with some non-zero real constant $c$, satisfying $f'(u)>0$. Then, the Gaussian curvature $K$ of $M^2$ satisfies \eqref{relation-K-f}, $1+K<0$, $\grad K\neq 0$ at any point of $M^2$, and the level curves of $K$ are circles of $M^2$ with positive constant signed curvature
\begin{equation*}
\kappa=\kappa(u)=\frac{3}{4}\frac{f'(u)}{f(u)}=-\frac{1}{4}\frac{|\grad K|}{K+f^2+1}.
\end{equation*}
\end{theorem} 

\begin{proof}
We have already seen in the proof of Theorem \ref{abstract-to-existence-immersion} that the Gaussian curvature $K=K(u)$ satisfies \eqref{relation-K-f}, $1+K<0$ and $\grad K\neq 0$ at any point of $M^2$. It remains to prove that the level curves of $K$ are circles of $M^2$ and find their curvature. In order to obtain this, we consider $\left\{\tilde{E}_1,\tilde{E}_2\right\}$ the positively oriented global orthonormal frame field on $M^2$ defined in the proof of Theorem \ref{abstract-to-existence-immersion}, i.e, 
$$
\tilde{E}_1=\frac{\grad K}{|\grad K|}=-\frac{\grad f}{|\grad f|}=-\frac{\partial}{\partial u}
$$
and
$$
\tilde{E}_2=-f^{3/4}(u)\frac{\partial}{\partial v}.
$$
Then, the Levi-Civita connection of $M^2$ is given by
$$
\nabla_{\tilde{E}_1}\tilde{E}_1=\nabla_{\tilde{E}_1}\tilde{E}_2=0, \quad \nabla_{\tilde{E}_2}\tilde{E}_1=\frac{3}{4}\frac{f'}{f}\tilde{E}_2, \quad \nabla_{\tilde{E}_2}\tilde{E}_2=-\frac{3}{4}\frac{f'}{f}\tilde{E}_1,
$$
and, since $K=K(u)$, it is clear that $\tilde{E}_1 K=-K'$ and $\tilde{E}_2K=0$. Thus, the integral curves of $\tilde{E}_2$ are the level curves of $K$. Also, we note that $\left\{\tilde{E}_2,-\tilde{E}_1\right\}$ is a positively oriented orthonormal frame field.

Now, let us consider $\gamma$ an integral curve of $\tilde{E}_2$, so $\gamma$ is a curve parametrized by arc-length. Then, if we define
\begin{equation*}
\kappa(u)= \frac{3}{4}\frac{f'(u)}{f(u)}>0
\end{equation*}
we obtain $\tilde{E}_2\kappa=0$,
$$
\nabla_{\tilde{E}_2}\tilde{E}_1=\kappa\tilde{E}_2, \qquad \text{ and } \qquad  \nabla_{\tilde{E}_2}\tilde{E}_2=\kappa\left(-\tilde{E}_1\right).
$$

As $\left\{\gamma', -\tilde{E}_{1|\gamma}\right\}$ is positively oriented and
$$
\nabla_{\gamma'}\left(-\tilde{E}_{1}\right)=-\kappa_{|\gamma}\gamma' \qquad \text{ and } \qquad \nabla_{\gamma'}\gamma'=\kappa_{|\gamma}\left(-\tilde{E}_{1|\gamma}\right) ,
$$
where $\kappa_{|\Gamma}$ is a constant, it follows that $\gamma$ is a circle of $M^2$ with positive constant signed curvature $\kappa$.
Moreover, using the link between $K$ and $f$, it is not difficult to see that
$$
\kappa=\frac{3}{4}\frac{|\grad f|}{f}=-\frac{1}{4}\frac{|\grad K|}{K+f^2+1}>0.
$$

\end{proof}

We have seen that if the metric $g$ is given by \eqref{expression-metricg-uv}, where $f$ is some positive solution of the second order ODE \eqref{second-order-chart-f}, then $f$ formally satisfies the Gauss equation \eqref{relation-K-f}. In order to obtain the converse of Theorem \ref{first-intrinsic-charact}, we need to change the interpretation of $f$. First, we give the next lemma.

\begin{lemma}\label{lemma:uniqueSolution}
	Let $\left(M^2,g\right)$ be an abstract surface such that $1+K<0$ and $c$ be a non-zero real constant. Then, the polynomial equation
	\begin{equation}\label{polynomialEquation}
		1+K+3f^2+c^2f^3=0
	\end{equation}
	and the condition $f>0$  uniquely determine $f$.
\end{lemma}

\begin{proof}
	Let $p_0\in M^2$ be an arbitrarily fixed point and consider the polynomial equation
	\begin{equation}\label{eqKx}
		K\left(p_0\right)=-1-3x^2-c^2x^3, \qquad x>0.
	\end{equation}
	Consider the function $h:(0,\infty)\to(-\infty, -1)$ given by 
	$$
	h(x)=-1-3x^2-c^2x^3.
	$$
	Clearly, $h$ is smooth and $h'(x)<0$ for any $x>0$. As
	$$
	\lim_{x\searrow 0}h(x)=-1 \qquad \text{and} \qquad \lim_{x\nearrow\infty} h(x)=-\infty,
	$$
	we get that $h$ is a smooth diffeomorphism. Therefore, the solution of \eqref{eqKx} is $x=h^{-1}\left(K\left(p_0\right)\right)$, i.e., $f\left(p_0\right)=h^{-1}\left(K\left(p_0\right)\right)$. Thus, $f=h^{-1}\circ K$ is a smooth function being the composition of two smooth functions, it is positive and the unique solution of \eqref{polynomialEquation}.
\end{proof}

As we already notice, if $\left(M^2,g\right)$ admits a (unique) $PNMC$ biconservative immersion, then $c$ and thus $f$ are unique. We will see that the same phenomenon also occurs in this intrinsic approach.

First, we state the converse of Theorem \ref{first-intrinsic-charact}.

\begin{theorem}\label{converse-first-charact}
Let $\left(M^2,g\right)$ be an abstract surface. Assume that the Gaussian curvature $K$ of $M^2$ satisfies $1+K<0$, $\grad K\neq 0$ everywhere and the level curves of $K$ are circles of $M^2$ with positive constant signed curvature $\kappa$ given by 
\begin{equation}\label{kappa-invariant}
\kappa=-\frac{1}{4}\frac{|\grad K|}{K+f^2+1},
\end{equation}
where $f$ is the positive solution of equation \eqref{polynomialEquation} for some non-zero real constant $c$. Then, around any point of $M^2$ there exist positively oriented local coordinates $(u,v)$ such that the metric $g$ is given by 
$$
g(u,v)=du^2+\frac{1}{f^{3/2}(u)}dv^2,
$$
and $f$ satisfies \eqref{second-order-chart-f} and $f'>0$.
\end{theorem}

\begin{proof}
First, let us consider $\left\{\tilde{E}_1,\tilde{E}_2\right\}$ the positively oriented global orthonormal frame field on $M^2$ defined by 
$$
\tilde{E}_1=\frac{\grad K}{|\grad K|}.
$$
It is clear that $\tilde{E}_2K=0$, i.e, the integral curves of $\tilde{E}_2$ are the level curves of $K$, so they are circles with curvature $\kappa$ given by \eqref{kappa-invariant}. Therefore,
$$
\nabla_{\tilde{E}_2}\tilde{E}_2=\kappa\left(-\tilde{E}_1\right), \qquad \text{ and } \qquad \nabla_{\tilde{E}_2}\left(-\tilde{E}_1\right)=-\kappa\tilde{E}_2.
$$
As $\tilde{E}_1=\grad K/ |\grad K|$, it is clear that $\tilde{E}_1 K=|\grad K|>0$ and $\tilde{E}_2 K=0$. Using \eqref{polynomialEquation}, it follows that $\tilde{E}_2f=0$ and then $\tilde{E}_2\left(\tilde{E}_1 K\right)=0$. Furthermore, as
\begin{align*}
\left[\tilde{E}_1,\tilde{E}_2\right]K & = \tilde{E}_1\left(\tilde{E}_2 K\right) -\tilde{E}_2\left(\tilde{E}_1 K\right)=0 \\
& = \left(\nabla_{\tilde{E}_1}\tilde{E}_2-\nabla_{\tilde{E}_2}\tilde{E}_1\right)(K), 
\end{align*}
it follows that $\left(\nabla_{\tilde{E}_1}\tilde{E}_2\right)(K)=0$, i.e., 
$$
\langle \nabla_{\tilde{E}_1}\tilde{E}_2,\grad K\rangle=0.
$$
Thus, $\langle \nabla_{\tilde{E}_1}\tilde{E}_2,\tilde{E}_1\rangle=0$. On the other hand, $\langle \nabla_{\tilde{E}_1}\tilde{E}_2,\tilde{E}_2\rangle=0$ and thus $\nabla_{\tilde{E}_1}\tilde{E}_2=0$. Further, it is easy to see that $\nabla_{\tilde{E}_1}\tilde{E}_1=0$.

Using \eqref{polynomialEquation} we rewrite the curvature $\kappa$ as
$$
\kappa=-\frac{3}{4}\frac{\tilde{E}_1 f}{f}.
$$
Using  again \eqref{polynomialEquation}, the Levi-Civita connection and the relation 
$$
d\omega_2^1\left(\tilde{E}_1,\tilde{E}_2\right)=K\left(\omega^1\wedge\omega^2\right)\left(\tilde{E}_1,\tilde{E}_2\right),
$$
by some straightforward computation we obtain 
\begin{equation}\label{secondDerivativefInvariant-tilde}
f\tilde{E}_1\left(\tilde{E}_1f\right) =\frac{7}{4}\left(\tilde{E}_1f\right)^2-\frac{4}{3}f^2-4f^4-\frac{4}{3}c^2f^5, 
\end{equation}
which can be seen that it is equivalent to equation \eqref{second-order-invariant-f}.

Let $p_0\in M^2$ be an arbitrarily fixed point of $M^2$ and let $\sigma=\sigma\left(v_2\right)$ be an integral curve of $\tilde{E}_2$ with $\sigma(0)=p_0$. Let $\left\{\phi_{u_2}\right\}_{u_2\in\mathbb{R}}$ be the flow of $\tilde{E}_1$. We define the following positively oriented local chart
$$
X^K\left(u_2,v_2\right)=\phi_{u_2}\left(\sigma\left(v_2\right)\right)=\phi_{\sigma\left(v_2\right)}\left(u_2\right).
$$
We have 
\begin{align*}
	&X^K\left(0,v_2\right)=\sigma\left(v_2\right),\\
	&X^K_{v_2}\left(0,v_2\right)=\sigma'\left(v_2\right)=\tilde{E}_2\left(\sigma\left(v_2\right)\right)=\tilde{E}_2\left(0,v_2\right), \\
	&X^K_{u_2}\left(u_2,v_2\right)=\phi'_{\sigma\left(v_2\right)}\left(u_2\right)=\tilde{E}_1\left(\phi_{\sigma\left(v_2\right)}\left(u_2\right)\right)=\tilde{E}_1\left(u_2,v_2\right)
\end{align*}
for any $\left(u_2,v_2\right)$. Clearly, $\left\{X^K_{u_2},X^K_{v_2}\right\}$ is positively oriented and as $\tilde{E}_1 K>0$, we also have $X^K_{u_2} K>0$.

If we write the Riemannian metric $g$ on $M^2$ in local coordinates as
$$
g=\langle \cdot, \cdot \rangle=g_{11}\ du_2^2+2g_{12}\ du_2\  dv_2+g_{22}\  dv_2^2,
$$
we get $g_{11}=\left|X^K_{u_2}\right|^2=\left|\tilde{E}_1\right|^2=1$, $g_{22}\left(0,v_2\right)=1$ and $g_{12}\left(0,v_2\right)=0$. Then, in a similar way, as in the proof of Theorem \ref{thm:fundamentalProperties}, we achieve
$$
\tilde{E}_2=\frac{1}{\sqrt{\det G}}\left(-g_{12}X^K_{u_2}+X^K_{v_2}\right),
$$
where 
\begin{equation*}
	G=G\left(u_2,v_2\right)=\left(
	\begin{array}{cc}
		g_{11} & g_{12} \\
		g_{12}  & g_{22}
	\end{array}
	\right).
\end{equation*}
Clearly, $\det G\left(0,v_2\right)=1$. As $\tilde{E}_2 K=0$ we get $X^K_{v_2} K=g_{12}X^K_{u_2}K$ and  $\tilde{E}_2\left(\tilde{E}_1 K\right)=0$.  By some standard computations we can prove that the last relation is equivalent to $X^K_{u_2}g_{12}=0$, so $g_{12}=g_{12}\left(v_2\right)=0$. Therefore, 
$$
\tilde{E}_{2}=\frac{1}{\sqrt{g_{22}}}X^K_{v_2}.
$$
Since $\tilde{E}_1K>0$ and $\tilde{E}_2K=0$, it follows that $X^K_{u_2} K>0$ and $X^K_{v_2}\ K=0$ and so
$$
K\left(u_2,v_2\right)=K\left(u_2,0\right)=K\left(u_2\right) \text{ with }K'\left(u_2\right)>0.
$$
Using a similar process as in the proof of Theorem \ref{thm:fundamentalProperties}, we achieve
$$
g_{22}\left(u_2,v_2\right)=\frac{\alpha\left(v_2\right)}{f^{3/2}\left(u_2\right)},
$$
where $\alpha$ is a smooth function strictly positive.

Thus, the metric $g$ can be written as
$$
g\left(u_2,v_2\right)=du_2^2+\frac{\alpha\left(v_2\right)}{f^{3/2}\left(u_2\right)}dv_2^2
$$
and 
$$
\tilde{E}_2=\tilde{E}_2\left(u_2,v_2\right)=\frac{f^{3/4}\left(u_2\right)}{\sqrt{\alpha\left(v_2\right)}}X^K_{v_2}.
$$
Further, we consider the change of coordinates 
$$
\left(u_2,v_2\right) \to \left(u=-u_2,v=-\int_{0}^{v_2} \sqrt{\alpha(\tau)}\ d\tau \right).
$$
In these new local coordinates, $f=f(u)$ satisfies $f'(u)>0$, the metric $g$ is given by
$$
g\left(u,v\right)=du^2+\frac{1}{f^{3/2}\left(u\right)}dv^2,
$$
and the vector fields $\tilde{E}_1$, $\tilde{E}_2$ are given by
\begin{equation}\label{E1E2-third chart}
\tilde{E}_1=-\frac{\partial}{\partial u} \qquad\text{ and }\qquad \tilde{E}_2=-f^{3/4}(u)\frac{\partial}{\partial v}.
\end{equation}
Replacing $\tilde{E}_1$ and $\tilde{E}_2$ from \eqref{E1E2-third chart} in \eqref{secondDerivativefInvariant-tilde}, by some straightforward computations we get that $f$ satisfies the second order ODE \eqref{second-order-chart-f}.
\end{proof}

\begin{remark}
From \eqref{kappa-invariant}, we see that $f$, and then $c$, are uniquely determined by the abstract surface $\left(M^2,g\right)$.
\end{remark}

From Theorems \ref{first-intrinsic-charact} and \ref{converse-first-charact} we conclude with the following characterization of the abstract surfaces $\left(M^2,g\right)$ with $g$ given by \eqref{expression-metricg-uv}.

\begin{theorem}\label{second-charact-bicons-surfaces-intrinsic}
Let $\left(M^2,g\right)$ be an abstract surface. Then, around any point of $M^2$ there exist positively oriented local coordinates $(u,v)$ such that the metric $g$ is given by 
$$
g(u,v)=du^2+\frac{1}{f^{3/2}(u)}dv^2,
$$
where $f=f(u)$ is some positive solution of the second order ODE \eqref{second-order-chart-f} with some non-zero real constant $c$, satisfying $f'(u)>0$, if and only if the Gaussian curvature $K$ of $M^2$ satisfies $1+K<0$, $\grad K\neq 0$ at any point of $M^2$ and the level curves of $K$ are circles of $M^2$ with positive constant signed curvature $\kappa$ given by 
\begin{equation*}
	\kappa=-\frac{1}{4}\frac{|\grad K|}{K+f^2+1},
\end{equation*}
where, now, $f$ is the positive solution of equation \eqref{polynomialEquation}.
\end{theorem}
 
A second characterization of the metric $g$ given by \eqref{expression-metricg-uv}, which actually follows from the above proofs, can be given in terms of its Levi-Civita connection. 

\begin{theorem}\label{second-intrinsic-theorem-levi-civita}
Let $\left(M^2,g\right)$ be an abstract surface. Then, around any point of $M^2$ there exist positively oriented local coordinates $(u,v)$ such that the metric $g$ has the expression
$$
g(u,v)=du^2+\frac{1}{f^{3/2}(u)}dv^2,
$$
where $f=f(u)$ is some positive solution of the second order ODE \eqref{second-order-chart-f} with some non-zero real constant $c$, satisfying $f'(u)>0$, if and only if the Gaussian curvature $K$ of $M^2$ satisfies $1+K<0$, $\grad K\neq 0$ at any point of $M^2$ and the Levi-Civita connection on $M$ is given by 
$$
\nabla_{\tilde{E}_1}\tilde{E}_1=\nabla_{\tilde{E}_1}\tilde{E}_2=0, \quad \nabla_{\tilde{E}_2}\tilde{E}_1=-\frac{3}{4}\frac{\tilde{E}_1 f}{f}\tilde{E}_2, \quad \nabla_{\tilde{E}_2}\tilde{E}_2=\frac{3}{4}\frac{\tilde{E}_1 f}{f}\tilde{E}_1,
$$
where  $\left\{\tilde{E}_1,\tilde{E}_2\right\}$ is the positively oriented global orthonormal frame field on $M^2$ defined by
$$
\tilde{E}_1=\frac{\grad K}{|\grad K|},
$$
and, now, $f$ is is the positive solution of equation \eqref{polynomialEquation}.
\end{theorem}

The last characterization of the abstract surfaces $\left(M^2,g\right)$ with $g$ given by \eqref{expression-metricg-uv}, is related to the curvature $\kappa$ of the level curves of the Gaussian curvature $K$. We show that $\kappa$ satisfies a third order ODE and it completely determines our abstract surfaces. In particular, it determines $f$ and $c$. We may say that the third order ODE describes the evolution of the level curves of $K$.

For the direct implication we can state the next theorem.

\begin{theorem}\label{third-characterization-first-implication}
Let $\left(M^2,g\right)$ be an abstract surface. Assume that around any point of $M^2$ there exist positively oriented local coordinates $(u,v)$ such that the metric $g$ is given by
$$
g(u,v)=du^2+\frac{1}{f^{3/2}(u)}dv^2,
$$
where $f=f(u)$ is some positive solution of the second order ODE \eqref{second-order-chart-f} with some non-zero real constant $c$, satisfying $f'(u)>0$. Let $\kappa$ be the positive signed curvature of the level curves of the Gaussian curvature $K$. Then,  
\begin{equation}\label{g-kappa}
	g(u,v)=du^2+\frac{1}{e^{2\int_0^u \kappa(\tau) d\tau}}dv^2.
\end{equation}
Moreover, $\kappa=\kappa(u)$ is the positive solution of the third order ODE 
\begin{equation}\label{eq-kappa}
3\kappa(u)\kappa'''(u)-26\kappa^2(u)\kappa''(u)-3\kappa'(u)\kappa''(u)+72\kappa^3(u)\kappa'(u)+32\kappa^3(u)-32 \kappa^5(u)=0
\end{equation}
and satisfies the conditions 
\begin{equation}\label{conditions-kappa}
\left\{
\begin{array}{l}
\kappa(u)>0 \\
\kappa'(u)<-1+\kappa^2(u) \\
2\kappa(u)\left(3\kappa'(u)-2\kappa^2(u)+2\right)<\kappa''(u)<\frac{2}{3}\kappa(u)\left(7\kappa'(u)-4\kappa^2(u)+4\right)
\end{array}
\right.,
\end{equation}
for any $u$.
\end{theorem}

\begin{proof}
Using Theorem \ref{second-charact-bicons-surfaces-intrinsic}, it follows that $1+K<0$, $\grad K\neq 0$ at any point of $M^2$, the level curves of $K$ are circles of $M^2$ with positive constant signed curvature $\kappa$ given by \eqref{kappa-invariant} and $f$ is also the positive solution of \eqref{polynomialEquation}.

Therefore, $|\grad K|=-K'(u)$ and we can rewrite \eqref{kappa-invariant} as
\begin{equation}\label{kappa-f'andf}
	\kappa=\kappa(u)=\frac{3}{4}\frac{f'(u)}{f(u)}.
\end{equation}
In order to obtain the ODE that is satisfied by $\kappa$, we make some standard computations. First, we compute $\kappa'$ and obtain 
\begin{align*}
\kappa'(u) & =\frac{3}{4}\left(\frac{f''(u)}{f(u)}-\left(\frac{f'(u)}{f(u)}\right)^2\right)\\
& =\frac{3}{4}\frac{f''(u)}{f(u)}-\frac{4}{3} \kappa^2(u).
\end{align*}
Thus,
\begin{equation}\label{f''/f}
\frac{f''(u)}{f(u)}=\frac{4}{3}\kappa'(u)+\frac{16}{9}\kappa^2(u).
\end{equation}
Then, from \eqref{second-order-chart-f} and \eqref{f''/f}, we get
\begin{equation}\label{first-expression-c}
c^2=\frac{\kappa^2(u)-\kappa'(u)-1}{f^3(u)}-\frac{3}{f(u)}.
\end{equation}
As the next step, deriving the above relation and using again \eqref{f''/f}, it follows
\begin{equation}\label{eq-useful f2}
-\kappa''(u)+6\kappa(u)\kappa'(u)+4\kappa(u)-4\kappa^3(u)+4\kappa(u) f^2(u)=0.
\end{equation} 
Further, if we derive \eqref{eq-useful f2} and use again \eqref{f''/f}, one gets
$$
-3\kappa'''(u)+18\kappa'^2(u)+18\kappa(u)\kappa''(u)+12\kappa'(u)-36\kappa^2(u)\kappa'(u)+4\left(3\kappa'(u)+8\kappa^2(u)\right)f^2(u)=0.
$$  
At the end, we replace $f^2(u)$ from \eqref{eq-useful f2} in the above equation and obtain \eqref{eq-kappa}.

In the following, we will obtain the conditions that $\kappa$ must satisfy.

First, as $f=f(u)$ and $f'=f'(u)$ are both positive functions, from \eqref{kappa-f'andf}, we get $\kappa=\kappa(u)>0$ for any $u$.

Second, if we integrate \eqref{kappa-f'andf} we get
\begin{equation*}\label{f-integral-kappa}
\frac{1}{f^{3/4}(u)}=\frac{A}{e^{\int_{0}^u \kappa(\tau)d\tau}},
\end{equation*}
where $A$ is a positive real constant. Performing a further simple change of coordinates, we can assume $A=1$ and the metric takes the form \eqref{g-kappa}. Therefore, we get
\begin{equation}\label{K-and-kappa}
K(u)=-\kappa^2(u)+\kappa'(u).
\end{equation} 
We know that $1+K(u)<0$, so 
\begin{equation*}
\kappa'(u)<-1+\kappa^2(u),
\end{equation*}  	
for any $u$.

Third, if we derive \eqref{K-and-kappa} and use $K'(u)<0$, we obtain
\begin{equation}\label{k''-first-inequality}
	\kappa''(u)<2\kappa(u)\kappa'(u),
\end{equation}  	
for any $u$.

We note that in the local coordinates $(u,v)$, equation \eqref{kappa-invariant} becomes
$$
\kappa(u)=\frac{K'(u)}{4\left(K(u)+f^2(u)+1\right)}.
$$  
Therefore, 
$$
f^2(u)=\frac{K'(u)-4\kappa(u)(1+K(u))}{4\kappa(u)}
$$
and using \eqref{K-and-kappa}, we achieve
\begin{equation}\label{link-f-kappa}
f^2(u)=\frac{\kappa''(u)-6\kappa(u)\kappa'(u)-4\kappa(u)+4\kappa^3(u)}{4\kappa(u)}>0,
\end{equation}
and therefore
\begin{equation*}
\kappa''(u)>2\kappa(u)\left(3\kappa'(u)-2\kappa^2(u)+2\right),
\end{equation*}
for any $u$.
We can replace $f(u)$ from \eqref{link-f-kappa} in \eqref{first-expression-c}, and obtain
\begin{equation}\label{second-expresion-c-kappa}
c^2=\frac{2\sqrt{\kappa(u)}\left(-3\kappa''(u)+14\kappa(u)\kappa'(u)+8\kappa(u)-8\kappa^3(u)\right)}{\left(\kappa''(u)-6\kappa(u)\kappa'(u)-4\kappa(u)+4\kappa^3(u)\right)^{3/2}}>0,
\end{equation}
and thus
\begin{equation*}
\kappa''(u)<\frac{2}{3}\kappa(u)\left(7\kappa'(u)-4\kappa^2(u)+4\right),
\end{equation*}
which is a stronger condition than \eqref{k''-first-inequality}, since $\kappa'(u)<-1+\kappa^2(u)$.

\end{proof}

We continue with the converse of Theorem \ref{third-characterization-first-implication}, obtaining in this way the third intrinsic characterization. First, inspired by the proof of Theorem \ref{third-characterization-first-implication}, we introduce the functions 
\begin{equation}\label{function c}
	\mathfrak{A}(u)= \frac{2\sqrt{\kappa(u)}\left(-3\kappa''(u)+14\kappa(u)\kappa'(u)+8\kappa(u)-8\kappa^3(u)\right)}{\left(\kappa''(u)-6\kappa(u)\kappa'(u)-4\kappa(u)+4\kappa^3(u)\right)^{3/2}},
\end{equation}
and
\begin{equation}\label{function-f^2}
\mathfrak{B}(u)= \frac{\kappa''(u)-6\kappa(u)\kappa'(u)-4\kappa(u)+4\kappa^3(u)}{4\kappa(u)},
\end{equation}
defined on an open interval which contains $0$. Now, we can prove the following result.
\begin{theorem}\label{third-characterization-converse-implication}
	Consider the third order ODE \eqref{eq-kappa} and the initial conditions
\begin{equation}\label{initial-conditions}
\left\{
\begin{array}{l}
\kappa(0)>0 \\
\kappa'(0)=\kappa_0' \\
\kappa''(0)=\kappa''_{0}
\end{array}
\right.
\end{equation}
such they satisfy 
\begin{equation*}
	\left\{
	\begin{array}{l}
		\kappa_0>0 \\
		\kappa'_0<-1+\kappa^2_0 \\
		2\kappa_0\left(3\kappa'_0-2\kappa^2_0+2\right)<\kappa''_0<\frac{2}{3}\kappa_0\left(7\kappa'_0-4\kappa^2_0+4\right)
	\end{array}
	\right..
\end{equation*}
Let $\kappa=\kappa(u)$ be the solution of \eqref{eq-kappa} and \eqref{initial-conditions} and assume that the inequalities \eqref{conditions-kappa} hold for any $u$ in some open interval containing $0$.
Consider $\left(M^2,g\right)$ an abstract surface with the metric $g$ given in some local coordinates $(u,v)$ by \eqref{g-kappa}.
Then, 
\begin{itemize}
\item [(i)] the Gaussian curvature of $M^2$ satisfies the relation
$$
K(u)=-\kappa^2(u)+\kappa'(u),
$$ 
and the conditions $1+K(u)<0$ and $\grad K\neq 0$ everywhere;
\item [(ii)] $\mathfrak{A}$ defined in \eqref{function c} is a real positive constant, $\mathfrak{A}=c^2$;  
\item[(iii)] $\mathfrak{B}$ defined in \eqref{function-f^2} is a positive function and if we denote by $f(u)=\sqrt{\mathfrak{B}(u)}$, we have
$$
K(u)=-1-3f^2(u)-c^2f^3(u);
$$
\item[(iv)] the level curves of $K$ are circles of $M^2$ with positive constant signed curvature $\kappa$, and $\kappa$ satisfies
$$
\kappa=-\frac{1}{4}\frac{|\grad K|}{K+f^2+1},
$$ 
where $f$ is given in $(iii)$.
\end{itemize}
\end{theorem}

\begin{proof}
Using the expression of the metric $g=g(u,v)$ given in \eqref{g-kappa}, by standard computation, we deduce
$$
K(u)=-\kappa^2(u)+\kappa'(u).
$$ 
From the second inequality in \eqref{conditions-kappa}, we get $1+K(u)<0$.

In order to prove that $\grad K\neq 0$ everywhere, i.e. $K'(u)\neq 0$ for any $u$, first we note that
$$
\frac{2}{3}\kappa(u)\left(7\kappa'(u)-4\kappa^2(u)+4\right)<2\kappa(u)\kappa'(u),
$$
for any $u$, since $\kappa'(u)<-1+\kappa^2(u)$. Then, using the third inequality in \eqref{conditions-kappa}, we obtain $\kappa''(u)<2\kappa(u)\kappa'(u)$, so 
$$
K'(u)=-2\kappa(u)\kappa'(u)+\kappa''(u)<0,
$$
for any $u$.

Further, from the first and third inequalities in \eqref{conditions-kappa} we obtain that $\mathfrak{A}$ and $\mathfrak{B}$ are positive functions. Now, using the expression of $K$ obtained in (i) and denoting $f(u)=\sqrt{\mathfrak{B}(u)}$, we can check that
\begin{equation}\label{alpha}
\mathfrak{A}(u)=-\frac{1+3f^2(u)+K}{f^3(u)}.
\end{equation}
By deriving the above relation we achieve
\begin{align*}
\mathfrak{A}'(u) & =- \left(3\kappa(u)\kappa'''(u)-26\kappa^2(u)\kappa''(u)-3\kappa'(u)\kappa''(u)+72\kappa^3(u)\kappa'(u)+32\kappa^3(u)-32 \kappa^5(u)\right)\\
&\qquad \cdot\frac{K'(u)}{32\kappa^3(u)f^5(u)}.
\end{align*}
Since $\kappa$ is a positive solution of \eqref{eq-kappa}, it follows that $\mathfrak{A}'(u)=0$, for any $u$, so $\mathfrak{A}$ is a positive constant and we denote it by $c^2$, where $c\neq 0$. With this notation, relation \eqref{alpha} can be rewritten as
$$
K(u)=-1-3f^2(u)-c^3f^3(u).
$$  
Next, we consider $\left\{\tilde{E}_1,\tilde{E}_2\right\}$ the positively oriented global orthonormal frame field on $M^2$ defined by
$$
\tilde{E}_1=\frac{\grad K}{|\grad K|}.
$$
We note that $\tilde{E}_2 K=0$, i.e., the integral curves of $\tilde{E}_2$ are the level curves of $K$. Moreover, since $|\grad K|=-K'(u)$, it is not difficult to see that 
$$
\tilde{E}_1=-\frac{\partial}{\partial u}\qquad \text{ and } \qquad \tilde{E}_2=-e^{\int_{0}^u\kappa(\tau)d\tau}\frac{\partial}{\partial v},
$$
and then
$$
\tilde{E}_2\kappa=0, \qquad \nabla_{\tilde{E}_2}\tilde{E}_1=\kappa\tilde{E}_2,\qquad \nabla_{\tilde{E}_2}\tilde{E}_2=\kappa\left(-\tilde{E}_1\right).
$$
Therefore, the level curves of $K$ are circles of $M^2$ with positive constant signed curvature $\kappa$.

As the last step of the proof, from (i), and then from  \eqref{polynomialEquation}, \eqref{link-f-kappa} and \eqref{second-expresion-c-kappa}, it follows that
$$
\frac{K'(u)}{4\left(K(u)+f^2(u)+1\right)}=\kappa(u),
$$ 
which is equivalent to 
$$
\kappa=-\frac{1}{4}\frac{|\grad K|}{K+f^2+1}.
$$ 
\end{proof}

\section{The codimension reduction for $PNMC$ biconservative surfaces in the hyperbolic space}

In the last section, we study the non-$CMC$ biconservative surfaces with parallel normalized mean curvature vector field in hyperbolic space $\mathbb{H}^n$, $n\geq 5$, and we prove that their substantial codimension is equal to $2$.

\begin{theorem} \label{reduction-theorem}
	Let $\varphi:\left(M^2,g\right)\to \mathbb{H}^n$, $n\geq 5$, be a $PNMC$ biconservative surface. Assume that the rank of the first normal bundle is at least $2$. Then, $M^2$ lies in some $4$-dimensional totally geodesic submanifold $\mathbb{H}^4$ of $\mathbb{H}^n$.
\end{theorem}

\begin{proof}
	Using the same notations as in previous sections, we recall that the biconservativity condition is equivalent to
	$$
	A_3\left(E_1\right)=-fE_1.
	$$
	From the expression of the mean curvature vector field
	$$
	H=\frac{1}{2}\left(\left(\trace A_3\right) E_3+\sum_{b=4}^{n}\left(\trace A_b\right) E_b\right)=fE_3,
	$$
	we infer that $\trace A_3=2f$ and $\trace A_b=0$, for any $b\in\left\{4,5,\ldots,n\right\}$. Further, as $\langle A_3\left(E_2\right),E_1 \rangle =0$, it follows that $A_3\left(E_2\right)=3fE_2$.
	
	Next, we will prove that $R^\perp(X,Y)\zeta=0$, for any $X,Y\in C\left(TM^2\right)$ and $\zeta\in C\left(NM^2\right)$. First, from the $PNMC$ hypothesis, we have that $\nabla^\perp E_3=0$ and so $R^\perp\left(E_1,E_2\right)E_3=0$. Then, using the Ricci equation, one gets $B\left(E_1,E_2\right)=0$
	and thus we obtain
	$$
	\langle A_b\left(E_1\right),E_2 \rangle =0, \qquad b\in\left\{4,5,\ldots,n\right\}.
	$$
	Consequently, $\left\{E_1,E_2\right\}$ simultaneously diagonalizes all the $A_\alpha$, $\alpha\in\left\{3,4,\ldots,n\right\}$, and therefore $M^2$ has flat normal bundle, i.e., $R^\perp(X,Y)\zeta=0$.
	
	Next, since $M^2$ has flat normal bundle and $\nabla^\perp E_3=0$, we can assume that $E_\alpha$ is parallel in $NM^2$, for any $\alpha \in\left\{3,4,\ldots,n\right\}$. Then, we set
	\begin{equation}
		\zeta_1:=B\left(E_1,E_1\right)=-fE_3+\sum_{b=4}^{n}f_bE_b
	\end{equation}
	and
	\begin{equation}
		\zeta_2:=B\left(E_2,E_2\right)=3fE_3-\sum_{b=4}^{n}f_bE_b,
	\end{equation}
	where $A_b\left(E_1\right)=f_bE_1$, for any $b\in\left\{4,5,\ldots,n\right\}$.
	Therefore, the first normal bundle of $M^2$ in $\mathbb{H}^n$ is given by
	$$
	N_1=\Span \Imag(B) = \Span\left\{B\left(E_1,E_1\right), B\left(E_2,E_2\right)\right\} = \Span\left\{\zeta_1,\zeta_2\right\}.
	$$
	From the hypothesis, we note that the rank of $N_1$ has to be equal to $2$ and all the $f_b$ cannot simultaneously vanish at any point.
	
	In the following we show that $N_1$ is a parallel with respect to the normal connection. In order to prove this, it is sufficient to show that $\nabla^\perp_{E_i}\zeta_j\in C\left(N_1\right)$, for any $i,j=1,2$.
	
	If we apply the Codazzi equation we obtain
	$$
	\nabla^\perp_{E_2}\zeta_1 = -\omega_2^1\left(E_1\right)\left(\zeta_1-\zeta_2\right)\in C\left(N_1\right)
	$$
	and 
	$$
	\nabla^\perp_{E_1}\zeta_2 = -\omega_2^1\left(E_2\right)\left(\zeta_1-\zeta_2\right)\in C\left(N_1\right).
	$$
	Moreover, as $\nabla^\perp E_\alpha=0$ and $E_2f=0$, we get
	$$
	\omega^1_2\left(E_1\right)=0, \quad E_2f_b=0, \quad E_1f=\frac{4}{3}\omega^1_2\left(E_2\right)f \quad \text{and} \quad E_1f_b=2\omega^1_2\left(E_2\right)f_b.
	$$
	Then, by direct computations, we infer
	$$
	\nabla^\perp_{E_2}\zeta_2 = 0\in C\left(N_1\right) \quad \text{and} \quad \nabla^\perp_{E_1}\zeta_1 = \frac{1}{3}\omega_2^1\left(E_2\right)\left(7\zeta_1+\zeta_2\right)\in C\left(N_1\right).
	$$
	Now, the codimension reduction result follows from \cite{E1971}.
	
\end{proof}

\end{document}